\documentclass[10pt]{amsart}
\usepackage{latexsym, amsmath, amssymb,amsthm,amsopn,amsfonts}
\usepackage{version}
\usepackage{epsfig,graphics,color,graphicx,graphpap}

\newcommand{\RR}{{\mathbb R}}

\newcommand{\SP}{{\mathbb S}}
\newcommand{\HS}{{\mathbb H}}

\theoremstyle{plain}

\newtheorem{thm}{Theorem}
\newtheorem{prop}{Proposition}[section]

\newtheorem{lem}[prop]{Lemma}

\theoremstyle{definition}

\newtheorem{rem}{Remark}[section]
\newtheorem{defn}[prop]{Definition} 

\numberwithin{equation}{section}

\def\squarebox#1{\hbox to #1{\hfill\vbox to #1{\vfill}}}

\usepackage{amsxtra}

\def\reals{{\mathbb R}}

\def\Ci{{\mathcal C}^\infty}

\def\O{{\mathcal O}}

\def\phi{\varphi}

\def\dist{\text{dist}\,}

\def\be{\begin{eqnarray*}}
\def\ee{\end{eqnarray*}}
\def\ben{\begin{eqnarray}}
\def\een{\end{eqnarray}}

\def\L2R{L_{\text{Rest}}^2}

\def\11{\mathds{1}}

\def\RR{\mathbb{R}}

\def\L2c{L^2_{\text{comp}}}

\def\tV{\widetilde{V}}

\def\tf{\tilde{f}}

\ifx\pdfoutput\undefined
  \DeclareGraphicsExtensions{.eps}
\else
  \ifx\pdfoutput\relax
    \DeclareGraphicsExtensions{.eps}
  \else
    \ifnum\pdfoutput>0
      \DeclareGraphicsExtensions{.pdf}
    \else
      \DeclareGraphicsExtensions{.eps}
    \fi
  \fi
\fi

\title
[Hyperbolic Solitons]
{Existence and stability of solitons for the nonlinear Schr\"odinger
  equation on hyperbolic space}

\author[H. Christianson]
{Hans Christianson}
\email{hans@math.mit.edu}
\address{Massachusetts Institute of Technology, Department of Mathematics\\
77 Mass. Ave., Cambridge, MA 02139-4307, USA}

\author[J.L. Marzuola]
{Jeremy L. Marzuola}
\email{jm3058@columbia.edu}

\address{Applied Mathematics Department, Columbia University \\
200 S. W. Mudd, 500 W. 120th St., New York City, NY 10027, USA}

\begin{document}    
   
\begin{abstract}
We study the existence and stability of ground state solutions or solitons to a nonlinear stationary equation on hyperbolic space.  The method of concentration compactness applies and shows that the results correlate strongly to those of Euclidean space.
\end{abstract}

\maketitle 

\section{Introduction}

In this note, we explore the existence of positive bound state solutions to the nonlinear Schr\"odinger equation on hyperbolic space ($\HS^d$-NLS)
\begin{eqnarray}
\left\{ \begin{array}{c}
iu_t + \Delta_{\HS^d} u + f(|u|) u = 0, \ \Omega \in \HS^d  \\
u(0,\Omega) = u_0 (\Omega),
\end{array} \right.
\label{eqn:hs}
\end{eqnarray}
where $\Delta_{\HS^d}$ is the (non-positive definite) Hyperbolic Laplacian and 
\begin{eqnarray*}
f(s) = s^{p}
\end{eqnarray*}
for $\frac{4}{d-2} > p > 0$.  Specifically, we seek to find solutions of the form
\begin{eqnarray*}
u(t,\Omega) = e^{i \lambda t} R_\lambda (\Omega) ,
\end{eqnarray*} 
where $R_\lambda$ is a solution for the resulting stationary problem
\begin{eqnarray}
\label{eqn:stat}
-\Delta_{\HS^d} R_\lambda + \lambda R_\lambda - R^{p+1}_\lambda = 0.
\end{eqnarray}
Our main result
is the following theorem.  Here and in the sequel, we adopt the
convention that for $d = 2$, $p< 4/(d-2)$ means $p < \infty$.

\begin{thm}
\label{thm:main-cm}
Fix $d \geq 2$.  For all $\lambda > (d-1)^2/4$, $\frac{4}{d-2} > p > 0$, there exists a solution to \eqref{eqn:hs} of the form
\begin{eqnarray*}
u(t,\Omega) = e^{i \lambda t} R_\lambda (\Omega),
\end{eqnarray*}
where $R_\lambda$ is a positive, decreasing, spherically symmetric solution of the equation \eqref{eqn:stat}.
\end{thm}

\begin{rem}
At the time of announcing this result, the authors have been informed
of a brief note by A. Pankov \cite{Pank} outlining a proof of a
similar result.
\end{rem}

\begin{rem}
The hypothesis $p < 4/(d-2)$ is the $H^1$-energy sub-critical regime.
The result in Theorem \ref{thm:main-cm} is precisely analogous to the existence of ground states
for Schr\"odinger equations on Euclidean space.  The key here is that
although the results are not drastically different in the case of
hyperbolic geometry, there are subtle difficulties that must be
overcome.  However, as will be seen in the sequel, with the correct
formulation the existence of solutions $R_\lambda$ will be almost
automatic due precisely to the nature of the background geometry at
infinity.  Hence one is able to show existence of such solutions for
in fact a rather larger class of nonlinearities which actually {\it
  grow} exponentially at infinity (see Section \ref{S:other-nl}).  
\end{rem}

\begin{rem}
In this note we analyze the existence and stability of ground state solutions, $R_\lambda$, but we say nothing about the uniqueness of such a solution.  There exists a very rich history of uniqueness proofs in the Euclidean case using shooting methods on the radial problem, which we believe should apply in this case as well.  For a survey of uniqueness results, see \cite{Mc}.  In addition, there are many interesting questions surrounding bound states once they are shown to exist, for instance the existence of a specific blow-up profile for a critical nonlinearity, if the hyperbolic geometry provides one with asymptotic stability for a wider range of nonlinearities due to the stronger dispersion, and many others.
\end{rem}

\subsection{Hyperbolic Space}

There are several equivalent definitions
of $\HS^d$.  The most intuitive is as an embedded hyperboloid in $\RR^{d+1}$:
\begin{eqnarray*}
\HS^d = \{ v = (v_0,v') \in \RR^{1+d} | \langle v, v \rangle = 1, \ v_{0} > 0 \},
\end{eqnarray*}
where
\begin{eqnarray*}
\langle v, v \rangle =  v_0^2 - | v'|^2 = v_0^2 - (v_1^2 + \cdots + v_d^2).
\end{eqnarray*}
The Lorenzian metric on $\RR^d$,
\begin{eqnarray*}
dl^2 = dv^2 - dv_0^2,
\end{eqnarray*}
restricts to a positive definite metric on $\HS^d$, called the
hyperbolic metric.  This manifold is isometric to the upper half space
\be
\{ x \in \reals^{d} : x_1 >0 \} 
\ee
equipped with the metric 
\be
ds^2 = \frac{dx^2}{x_1^2},
\ee
as well as to the Poincar\'e ball model:
\be
\HS^d = \{z \in \reals^d : |z| <1 \}
\ee
with the metric
\be
ds^2 = \frac{4 dz^2}{(1-|z|^2)^2}.
\ee
In this note, we wish to exploit certain spherical symmetries, so we
use the polar model:
\begin{eqnarray*}
\HS^d = \{ (t,x) \in \RR^{1+d} | (t,x) = ( \cosh(r),\sinh(r) \omega), r \geq 0, \omega \in \SP^{d-1} \}.
\end{eqnarray*}
From the polar coordinate parametrization of $\HS^d$, we see
\begin{eqnarray*}
dt = \sinh(r) dr, \ dx = \cosh(r) \omega dr + \sinh(r) d\omega, 
\end{eqnarray*}
giving the metric
\begin{eqnarray*}
ds^2 = dr^2 + \sinh^2 r d\omega^2
\end{eqnarray*}
from the standard Lorenzian metric restricted to $\HS^d$.  
Hence, we see
\begin{eqnarray}
\label{E:Delta-HS}
\Delta_{\HS^d} = \partial_r^2 + (d-1) \frac{\cosh r}{\sinh r} \partial_r + \frac{1}{\sinh^2 r} \Delta_{\SP^{d-1}}.
\end{eqnarray}

It is a standard exercise in differential geometry to show hyperbolic
space has constant sectional curvatures all equal to $-1$.

\subsection{Sketch of the proof}

In hyperbolic space, we have the conserved quantities
\begin{eqnarray*}
Q (u) = \int_{\HS^d} | u |^2 d \Omega
\end{eqnarray*}
and
\begin{eqnarray}
\label{E:hyp-E}
E(u) = \int_{\HS^d} \left( | \nabla_{\HS^d} u |^2 - \frac{2}{p+2} |u|^{p+2} \right) d\Omega,
\end{eqnarray}
and the approach of this paper is to realize solitons as minimizers
for certain constrained minimization problems related to these
quantities, and mimic the arguments used in the Euclidean setting.  

We use the polar representation of $\HS^d$, in which case the
hyperbolic Laplacian $\Delta_{\HS^d}$ can be conjugated to the
Euclidean Laplacian $\Delta_{\reals^d}$, modulo a potential term and
an angular offset term (see Section \ref{S:red-euc} below).  After
this conjugation, we are left with an equivalent optimization problem
in Euclidean space (see \eqref{eqn:E}).  This minimization problem has
an awkward angular term, so it is greatly simplified by assuming
spherical symmetry of a minimizing sequence.  To prove this
simplification is justified, we first prove any minimizing sequence of
the problem {\it in hyperbolic space} may be replaced by one that is spherically symmetric.  Conjugating the problem to Euclidean space amounts to replacing the
minimizing sequence with the sequence multiplied by a positive, radial
function, so conjugation {\it preserves the spherical symmetry}.  Then
we study the minimization problem in Euclidean space assuming
spherical symmetry, in which case it is equivalent to a minimization
problem with the standard Euclidean Laplacian.  Finally, that the
minimization problem may be reduced to the spherically symmetric case
follows from a rearrangement inequality from \cite{Dra} presented in Section \ref{sec:rad}.

\subsection{Acknowledgments}  
H. C. was partially supported by an NSF Postdoctoral Fellowship while in
residence at the Mathematical Sciences Research Institution (MSRI),
program ``Analysis on Singular Spaces''.  J.M. was partly supported by a National Science Foundation Postdoctoral Fellowship at Columbia University and partly by the Hausdorff Center for Mathematics at the University of Bonn.  The authors wish to thank Gigliola Staffilani, Herbert Koch, Michael Taylor, and Michael Weinstein for helpful conversations, as well as Vedran Sohinger and the reviewers for a careful reading of the draft and helpful suggestions. In addition, the second author would like to thank Jason Metcalfe and the University of North Carolina, Chapel Hill for graciously hosting him during part of this research.

\section{Previous Results}
\label{S:previous-results}

In this section, we summarize known results for soliton existence in $\RR^d$ and some of the recent work on HNLS.  

We first recall the relevant definitions for solitons in Euclidean
space.  Let $u(t,x)$ be a solution to the following Euclidean
nonlinear Schr\"odinger equation (NLS):
\begin{eqnarray}
\left\{ \begin{array}{c}
iu_t + \Delta_{\reals^d} u + f(|u|) u = 0, \ x \in \reals^d  \\
u(0,x) = u_0 (x),
\end{array} \right.
\label{eqn:nls}
\end{eqnarray}
where $\Delta_{\reals^d}$ is the (non-positive definite) Laplacian and 
\begin{eqnarray*}
f(s) = s^{p}
\end{eqnarray*}
for $\frac{4}{d-2} > p > 0$.

A soliton solution in Euclidean space is of the form 
\begin{eqnarray*}
u(t,x) = e^{i \lambda t} R_\lambda(x)
\end{eqnarray*} where 
$\lambda > 0$ and $R_\lambda (x)$ is a positive, spherically symmetric, exponentially decaying solution of the equation
\begin{eqnarray}
\label{eqn:sol}
\Delta R_\lambda - \lambda R_\lambda + f (R_\lambda ) R_\lambda = 0.
\end{eqnarray}
There are two conserved quantities for sufficiently regular solutions $u$ to NLS:
\begin{eqnarray*}
Q (u) := \int_{\reals^d} | u |^2 d x
\end{eqnarray*}
and
\begin{eqnarray*}
E(u) := \int_{\reals^d} \left( | \nabla u |^2 - F(|u|) \right) dx,
\end{eqnarray*}
where $F(|u| ) = | u |^{p+2}/(p+2)$ (more general nonlinearities can
also be considered by replacing $F$ with the integral of $f$; see
Section \ref{S:other-nl} for a discussion of other nonlinearities in the case of
hyperbolic space studied in this paper).  With this type of power
nonlinearity, soliton solutions exist and are known to be unique.
Existence of solitary waves for a wide variety of nonlinearities is
proved in \cite{BeLi} by {\it minimizing} the quantity
\begin{eqnarray*}
T(u) = \int_{\reals^d} |\nabla u|^2 dx
\end{eqnarray*}
with respect to the constraint
\begin{eqnarray*}
V(u) := -\frac{\delta^2}{2} \int_{\reals^d} |u|^2 dx + \int_{\reals^d} F(|u|) dx  = 1.
\end{eqnarray*}   
Then, using a minimizing sequence and Schwarz symmetrization, one sees the existence of the nonnegative, spherically symmetric, decreasing soliton solution.  For uniqueness, see \cite{Mc}, where a shooting method is implemented to show that the desired soliton behavior only occurs for one particular initial value.

An important fact for these soliton solutions is that $Q_{\lambda} = Q(R_{\lambda})$ and $E_{\lambda} = E(R_{\lambda})$ are differentiable with respect to $\lambda$.  This fact can be determined from the early works of Shatah, namely \cite{Sh1}, \cite{Sh2}.  By differentiating Equation \eqref{eqn:sol}, $Q$ and $E$ with respect to $\lambda$, we have
\begin{eqnarray*}
\partial_{\lambda} E_{\lambda} = - \lambda \partial_{\lambda} Q_{\lambda}.
\end{eqnarray*}

Variational techniques developed by \cite{W1} and \cite{W2} and generalized in \cite{GSS} and \cite{ShSt} tell us that when $\beta ( \lambda ) = E_{\lambda} + \lambda Q_{\lambda}$ is convex, or $\beta '' (\lambda) > 0$, we are guaranteed orbital stability as will be defined in the sequel (see Section \ref{sec:props}) under small perturbations, while for $\beta '' (\lambda) < 0$ we are guaranteed that the soliton is unstable under small perturbations.  

In this note we expand these results on bound states to $\HS^d$, following the work of Banica on well-posedness for focusing-HNLS in \cite{Ban1}.  The subsequent works of Banica-Carles-Staffilani (\cite{BCS}), Banica-Carles-Duyckaerts (\cite{BCD}), and Ionescu-Staffilani (\cite{IS}) study the questions of global well-posedness of the defocusing and focusing-HNLS, though methods there apply broadly to questions of local well-posedness.  In both the focusing and defocusing cases, the results parallel the Euclidean space results quite well, especially in dimension $3$.  We recall the results below only for $\HS^d$, but for a collection of comparable results in $\RR^d$, see the references contained within \cite{Ban1}, \cite{BCS}, \cite{IS} or a general presentation of the theory is done quite nicely in the book by Sulem-Sulem (\cite{SS}).

In \cite{Ban1}, the following theorem is proved, which states roughly
that well-posedness (existence, uniqueness and Lipschitz dependence
upon initial data) in $H^1$ and theory of blow-up (typically
classified as a singularity for the quantity $\int | \nabla u|^2 dx$
reached in finite time) for \eqref{eqn:hs} are comparable to the
results for the focusing, monomial nonlinear Schr\"odinger equation in
$\RR^d$.  Specifically, it is stated that for $p < \frac{4}{d}$, there
is global well-posedness and for $p \geq \frac{4}{d}$ there is local
well-posedness but also the possibility of finite time blow-up. These
statements are collected in the following theorem. 

\begin{thm}[Banica]
For $p < \frac{4}{d}$ the solutions to equation \eqref{eqn:hs} with $f(s) = s^{p}$ are global in $H^1 (\HS^d)$.  Global existence still holds for the power $p = \frac{4}{d}$ with initial data of mass smaller than a certain constant.

However, for $p \geq \frac{4}{d}$, blow-up solutions exist.  More precisely, if the initial data is radial and of finite variance
\begin{eqnarray*}
\int_{\HS^d} |u_0 (\Omega)|^2 \text{dist}^2 (0,\Omega) d\Omega < \infty
\end{eqnarray*}
and its energy satisfies
\begin{eqnarray*}
E(u_0) < c_d \| u_0 \|_{L^2}^2,
\end{eqnarray*}
then the solution blows up in finite time.  Here, $c_d$ is a geometric positive constant given by
\begin{eqnarray*}
c_d = \frac{ \inf \Delta^2_{\HS^d} \text{dist}(0,\cdot)}{16}.
\end{eqnarray*}
\end{thm}
This theorem allows for blow-up even for null energy solutions, which differs from the standard Glassey-type blow-up results for NLS on Euclidean space.

Recall that {\it scattering} to $u_\pm$ for a solution to HNLS means that such a solution $u(t,x)$ satisfies
\begin{eqnarray*}
\| u(t) - e^{i (t-t_0) \Delta_{\HS^d}} u_\pm \|_{L^2} \to 0 \ \text{as} \ t
\to \pm \infty. 
\end{eqnarray*}
The idea is that asymptotically the nonlinear problem is essentially
controlled by the linear component.  A {\it wave operator} $W_{\pm}$
is a well-defined map from the scattering data to data at time $t_0$:
\[
 W_{\pm} u_\pm = u(t_0).
\]
Wave operators are injective by uniqueness, and we say {\it
  asymptotic completeness} occurs if they are also surjective.
Obviously, this is not possible in the focusing case due to the
existence of bound states as described in this paper.

For the defocusing equation, we state the following theorems here as the techniques used to prove Theorem \ref{thm:bcs1} apply to the focusing problem when analyzing local well-posedness.  As a result, they are applicable when proving persistence of radiality for solutions of HNLS.

In \cite{BCS}, the following theorem about scattering for the defocusing, monomial nonlinear Schr\"odinger equation on $\HS^d$ is proved.  

\begin{thm}[Banica-Carles-Staffilani]
\label{thm:bcs1}
Let $d \geq 2$, $0 < p < \frac{4}{d}$ and $t_0 \in \RR$.  There exists
$\epsilon = \epsilon (d,p)$ such that if $\phi \in L^2_{r} (\HS^d)$
with $\| \phi \|_{L^2} < \epsilon$, then \eqref{eqn:hs} with $f(s) =
-s^{p}$ and $u|_{t = t_0} = \phi$ has a solution
\begin{eqnarray*}
u \in \mathcal{C} (\RR; L^2) \cap L^{2+p} (\RR \times \HS^d).
\end{eqnarray*}
Moreover, $\| u(t) \|_{L^2} = \| \phi \|_{L^2}$ for all $t \in \RR$.  There exist $u_{\pm} \in L^2_{r} (\HS^d)$ such that
\begin{eqnarray*}
\| u(t) - e^{i( t-t_0) \Delta_{\HS^d}} u_\pm \|_{L^2} \to 0 \ \text{as} \ t \to \pm \infty.
\end{eqnarray*}
If we take initial time $t_0 = -\infty$ (resp. $t_0 = \infty$), then
$u_- = \phi$
(resp. $u_+ = \phi$).
\end{thm}

Also in \cite{BCS}, the following theorem is proved.  
\begin{thm}
Let $d \geq 2$, $0 < p < \frac{4}{d-2}$ and $t_0 = -\infty$.  For any $\phi = u_- \in H^1_{r} (\HS^d)$, there exists $T < \infty$ such that the Cauchy integral formulation of \eqref{eqn:hs} with $f(s) = -s^{p}$ has a solution
\begin{eqnarray*}
u \in \mathcal{C} (\RR; L^2) \cap L^\infty ([-\infty, -T];H^1) \cap L^{2+p} ([-\infty,-T]; W^{1,p+2}).
\end{eqnarray*}
Moreover, this solution $u$ is defined globally in time; $u \in L^\infty (\RR;H^1)$.  That is, $u$ is the only solution to \eqref{eqn:hs} with $f(s) = -s^{p}$ such that
\begin{eqnarray*}
\| u(t) - e^{i (t-t_0) \Delta_{\HS^d}} u_- \|_{H^1} = \| e^{-i (t-t_0) \Delta_{\HS^d}} u(t) - u_- \|_{H^1} \to 0 \ \text{as} \ t \to \infty.
\end{eqnarray*}
\end{thm}
\noindent In addition, the authors prove $H^1$ asymptotic completeness in the case $d=3$.

In \cite{IS}, the following theorem is proved.
\begin{thm}
Let $d \geq 2$, $0 < p < \frac{4}{d-2}$ and $q \in (2, (2d+4)/d]$ is fixed.

a.  If $u_0 \in H^1 (\HS^d)$ then there exists a unique global solution $u \in \mathcal{C} (\RR; H^1 (\HS^d))$ of \eqref{eqn:hs} with $f(s) = -s^{p}$.  In addition for $T \in [0,\infty)$, the mapping
\begin{eqnarray*}
u_0 \to e^{i T \Delta_{\HS^d}} (u_0) = 1_{(-T,T)} (t) \cdot u
\end{eqnarray*}
is a continuous mapping from $H^1$ to $S^1_q (-T,T)$ and the conservation laws are satisfied, where
\begin{eqnarray*}
S^1_q (I) = \left\{ f \in C(I:L^2 (\HS^d)) : \| f \|_{S^1_q} = \| (-\Delta_{\HS^d})^{\frac{1}{2}} (f) \|_{S^0_q (I)} < \infty \right\} , 
\end{eqnarray*}
\begin{eqnarray*}
S^0_q (I) = \left\{ f \in C(I:H^1 (\HS^d)) : \| f \|_{S^0_q} = \sup \left[ \| f \|_{L^{\infty,2}_I},  \| f \|_{L^{q,r}_I},  \| f \|_{L^{q,q}_I} \right] < \infty \right\} ,
\end{eqnarray*}
and
\begin{eqnarray*}
r = \frac{2 d q}{dq-4}.
\end{eqnarray*}

b.  Asymptotic completeness occurs in $H^1$.
\end{thm}

Similar scattering results are obtained in the concurrent work \cite{AnPi}.  It should be noted that in the case of defocussing nonlinearities in Euclidean space, scattering is only proved for $L^2$-supercritical but $H^1$-subcritical powers, making the result far stronger for defocussing HNLS.

\section{Radiality Assumption}
\label{sec:rad}

As mentioned in the introduction, the proof of Theorem
\ref{thm:main-cm} relies on conjugating $\Delta_{\HS^d}$ into an
operator on Euclidean space, and then finding minimizers for the
energy functional in \eqref{eqn:E}.  The problem of minimizing the functional \eqref{eqn:E} is greatly
simplified assuming the functions involved depend only on the radius
$r = |x|$, as then the minimization theory in $\RR^d$ may be used,
since the term involving the angular derivatives vanishes.  The purpose of this section is to justify such a
simplification.  
Let us define a space $H^1_r$ to be the space of all spherically
symmetric functions in $H^1$.

The next lemma shows that spherically symmetric initial data implies a
spherically symmetric solution to HNLS.

\begin{lem}
\label{lem:rad}
Let $u$ be a solution to \eqref{eqn:hs} with initial data $u_0 \in
H^1_r$ and the nonlinearity $f(|u|)u = |u|^p u$ with $\frac{4}{d-2} >
p > 0$.  Then $u \in H^1_r$.
\end{lem}
The proof of this lemma is by uniqueness, which follows from 
 the implicit local uniqueness following from the Strichartz estimates
 in \cite{IS}.

Given Lemma \ref{lem:rad}, we show that any minimizer of \eqref{E:hyp-E} may be replaced by one that is spherically symmetric, so that we may neglect the angular derivative.  To do this, we modify the standard
argument of \cite[Lemma 7.17]{LL} in $\RR^d$, using heat kernel arguments to show symmetric decreasing rearrangement or Schwarz symmetrization lowers the kinetic energy in $\HS^d$.  The symmetric decreasing rearrangement on $\HS^d$ is given by
\begin{eqnarray*}
f^* (\Omega) = \inf \{ t: \lambda_f (t) \leq \mu (B(\dist(\Omega, 0))) \},
\end{eqnarray*}
where $\mu$ is the natural measure on $\HS^d$, $\dist$ is the hyperbolic distance function on $\HS^d$ and
\begin{eqnarray*}
\lambda_f (t) = \mu ( \{ |f| > t \} ).
\end{eqnarray*}
First of all, it is clear $f^*$ is spherically symmetric, nonincreasing, lower semicontinuous and 
\begin{eqnarray*}
\| f^* \|_{L^p (\HS^d)} = \| f \|_{L^p (\HS^d)}
\end{eqnarray*}
for any $1 \leq p \leq \infty$.

\begin{lem}
\label{lem:rad1}
Suppose $f \in H^1( \HS^d)$, and $f^*$ is the symmetric decreasing
rearrangement of $f$.  Then
\be
\| \nabla f^* \|_{L^2(\HS^d)} \leq \| \nabla f \|_{L^2(\HS^d)}.
\ee
\end{lem}
\begin{proof}
We use standard Hilbert space theory as in \cite{LL}.  Namely, we
observe that the kinetic energy satisfies
\begin{eqnarray*}
\| \nabla f \|_{L^2 (\HS^d)} = \lim_{t \to 0} I^t (f),
\end{eqnarray*}
where
\begin{eqnarray*}
I^t (f) = t^{-1} [ (f,f)_{\HS^d} - (f,e^{\Delta_{\HS^d} t} f)_{\HS^d}]
\end{eqnarray*}
and $(\cdot,\cdot)_{\HS^d}$ is the natural $L^2$ inner-product on $\HS^d$.
As $(f,f) = (f^*,f^*)$ by construction, we need
\begin{eqnarray*}
(f^*,e^{\Delta_{\HS^d} t} f^*)_{\HS^d} \geq (f,e^{\Delta_{\HS^d} t} f)_{\HS^d} 
\end{eqnarray*}
in order to see that symmetrization decreases the kinetic energy.
In $\RR^d$, this is done using convolution operators and the Riesz
rearrangement inequality, which we do not have here.  
Instead, we use Lemma \ref{HK-dec} and an application of Theorem \ref{Dra-thm} with $\Psi (f_1,f_2) = f_1 f_2$ and $K_{12} = p_d ( \rho, t)$ to finish the
proof of the lemma.
\end{proof}

\begin{lem}
\label{HK-dec}
For each $t>0$, the heat kernel on hyperbolic space, $p_d(\rho,t)$, is
a decreasing function of the hyperbolic distance $\rho$.
\end{lem}

\begin{proof}
This follows from Proposition 3.1 and the recursion relations in
Theorem 2.1 in \cite{DM}.  Specifically,
$\HS^1$ is isometric to $\reals$ with the metric $dx^2$, so the heat
kernels are the same:
\be
p_1(\rho,t) = (4 \pi t)^{-1/2} e^{-\rho^2/4t},
\ee
and we have the recurrence relations (see \cite[Theorem 2.1]{DM})
\ben
p_{d+1}( \sigma,t) & = & - (4 \pi)^{-1} \frac{\partial}{\partial \sigma}
p_{d-1}(\sigma,t) \text{ and } \label{pd-1} \\
\frac{1}{2} p_d (\sigma,t) & = &  \int_\sigma^\infty p_{d+1}( \lambda,t) (\lambda - \sigma )^{-1/2} d \lambda
, \label{pd-2}
\een
where $\sigma$ is related to the hyperbolic distance $\rho$ by 
\be
\sigma = [\cosh (\rho/2)]^2.
\ee
Since $\cosh$ is a monotone increasing function for $\rho \geq 0$, it suffices to prove
the lemma with $\sigma$ in place of $\rho$.  
Further, from \cite[Proposition 3.1]{DM}, we have $p_d \geq 0$ for $d$
odd, so \eqref{pd-2} implies $p_d \geq 0$ for $d$ even as well.  Then
\eqref{pd-1} indicates that the derivative of $p_{d-1}$ is a negative
multiple of $p_{d+1}$, and hence is negative.
\end{proof}

From \cite{Dra}, we have used the following theorem.

\begin{thm}[Draghici \cite{Dra}]
\label{Dra-thm}
Let $X = \HS^d$, $f_i : X \to \RR_+$ be $m$ nonnegative functions, $\Psi \in AL_2 (\RR^m_+)$ be continuous and $K_{ij}: [0,\infty) \to [0,\infty)$, $i < j$, $j \in \{ 1, \dots, m \}$ be decreasing functions.  We define
\begin{eqnarray*}
I[f_1, \dots, f_m] = \int_{X^m} \Psi (f_1 (\Omega_1), \dots, f_m (\Omega_m)) \Pi_{i<j} K_{ij} (d(\Omega_i, \Omega_j)  ) d\Omega_1 \dots d\Omega_m.
\end{eqnarray*}
Then, the following inequality holds:
\begin{eqnarray*}
I[f_1, \dots, f_m] \leq I[f_1^*, \dots, f_m^*].
\end{eqnarray*}
\end{thm}

\section{Reduction to an Euclidean Operator}
\label{S:red-euc}

In this section we begin to analyze HNLS in the case $f(|u|)u$ is a
so-called ``focusing'' nonlinearity.  From the polar form of $\Delta_{\HS^d}$, we approach the problem by comparison to the standard Laplacian on $\RR^d$.  In this direction, let us recall that the metric for $\RR^d$ in polar coordinates is given by
\begin{eqnarray*}
ds^2 = dr^2 + r^2 d\omega^2,
\end{eqnarray*}
so that the Jacobian is $r^{d-1}$.  Similarly, the Jacobian from the polar coordinate representation of $\HS^d$ is
$\sinh^{d-1} r$.  We employ an isometry $T$ taking $L^2(r^{d-1} dr
d\omega)$ to  $L^2(\sinh^{d-1} r dr d\omega)$, so that $T^{-1}
(-\Delta_{\HS^d}) T$ is a non-negative, unbounded, essentially self-adjoint operator on $L^2( \reals^d)$.

We define
\begin{eqnarray*}
\phi(r) = \left( \frac{r}{\sinh r} \right)^{\frac{d-1}{2}}, \\
\phi^{-1} (r) = \left( \frac{\sinh r}{r} \right)^{\frac{d-1}{2}},
\end{eqnarray*}
and take $Tu = \phi u$.  
Conjugating $-\Delta_{\HS^d}$ by $\phi$, we have a second order
differential operator on $\RR^d$ with the leading order term {\it
  almost} the Laplacian on $\RR^d$.  Indeed, we first calculate
\begin{eqnarray*}
\partial_r \phi & = & \frac{d-1}{2} \left( \frac{r}{\sinh r} \right)^{\frac{d-3}{2}} \left( \frac{\sinh r - r \cosh r}{\sinh^2 r } \right) , \\
\partial_r^2 \phi & = & \left( \frac{d-1}{2} \right) \left( \frac{d-3}{2} \right)  \left( \frac{r}{\sinh r} \right)^{\frac{d-5}{2}} \left( \frac{\sinh r - r \cosh r}{\sinh^2 r } \right)^2 \\
& + & \frac{d-1}{2} \left( \frac{r}{\sinh r} \right)^{\frac{d-3}{2}}  \left( \frac{2r \sinh r \cosh^2 r - 2 \sinh^2 r \cosh r -r \sinh^3 r}{\sinh^4 r } \right), 
\end{eqnarray*}
so that
\begin{eqnarray*}
\phi^{-1} (-\Delta_{\HS^d}) (\phi u ) & = & \phi^{-1} (-\partial_r^2 - (d-1) \frac{\cosh r}{\sinh r} \partial_r - \frac{1}{\sinh^2 r} \Delta_{\SP^{d-1}}) (\phi u) \\
& = & -\partial_r^2 u - 2 \phi^{-1} \partial_r \phi \partial_r u - \phi^{-1} \partial_r^2 \phi u - (d-1) \frac{\cosh r}{\sinh r} \partial_r u - (d-1) \frac{\cosh r}{\sinh r} \phi^{-1} \partial_r \phi u \\
&  & -\frac{1}{\sinh^2 r} \Delta_{\SP^{d-1}} u \\
& = & -\partial_r^2 u - \left(2 \phi^{-1} \partial_r \phi + (d-1) \frac{\cosh r}{\sinh r}\right) \partial_r u - \frac{1}{\sinh^2 r} \Delta_{\SP^{d-1}} u  \\
&  & -\left(\phi^{-1} \partial_r^2 \phi + (d-1) \frac{\cosh r}{\sinh r} \phi^{-1} \partial_r \phi \right) u \\
& = & -\partial_r^2 u + V_0(r) \partial_r u + \left[ V_d(r)  + \left( \frac{d-1}{2} \right)^2 \right] u - \frac{1}{\sinh^2 r} \Delta_{\SP^{d-1}} u \\
& = & -\tilde{\Delta} u + \left[ V_d(r)  + \left( \frac{d-1}{2} \right)^2 \right] u.
\end{eqnarray*}
Here
\begin{eqnarray*}
V_0(r) & = & \frac{1-d}{r}, \\
V_d (r) & = &  \left( \frac{d-1}{2} \right) \left( \frac{d-3}{2}
\right) \frac{1}{\sinh^2 r} - \left( \frac{d-1}{2} \right) \left( \frac{d-3}{2} \right) \frac{1}{r^2}  \\
& = & \frac{(d-1)(d-3)}{4} \left( \frac{r^2 - \sinh^2 r }{r^2 \sinh^2
    r}\right),
\end{eqnarray*}
so that
\be
- \tilde{\Delta} = - \Delta_{\RR^d} - \frac{r^2 - \sinh^2 r}{r^2
  \sinh^2 r} \Delta_{\SP^{d-1}}.
\ee

For completeness, we record the following
simple lemma.
\begin{lem}
The function 
\be
\tV =  \frac{\sinh^2 r - r^2}{r^2 \sinh^2 r}
\ee
satisfies the following properties: 
\be
&& \text{  (i) } \tV \in \Ci (\reals ), \\
&& \text{ (ii) } \tV \geq 0, \\
&& \text{(iii) } \tV(0) = \frac{1}{3}, \\
&& \text{ (iv) } \tV = \O( r^{-2}), \,\, r \to \infty, \text{ and} \\
&& \text{  (v) } \tV'(r) = 0 \text{ only at } r = 0.
\ee
\end{lem}
\begin{rem}
Note that the potential $V_3 = 0$, and the lemma implies $V_2 \geq 0$
has a ``bump'' at $0$, while for $d \geq 4$, the 
potential $V_d \leq 0$ has a ``well'' at $0$. 
\end{rem}

\begin{proof}
Properties (i), (ii), and (iii) follow easily from Taylor expansions
and the fact that $\sinh r \geq r$ for $r \geq 0$.  To prove the only
critical point is the origin, we observe $\tV'(r) = 0$ if and only if
\be
\frac{\cosh r }{\sinh^3 r} = \frac{1}{r^3},
\ee
so we consider 
\be
\frac{ \sinh r}{ \cosh^{1/3} r} = r.
\ee
As this equation is satisfied for $r = 0$, if we can show the
derivative of the left hand side is greater than $1$ for $r >0$ we are
done.  Differentiating the left hand side, setting it equal to $1$ and
rearranging we have the equation
\be
8 \cosh^6 r - 15 \cosh^4 r + 6 \cosh^2 r +1 = 0.
\ee
Substituting $z = \cosh^2 r$, we have the third order polynomial
equation
\be
8 z^3 - 15 z^2 + 6 z +1 = 0,
\ee
which factorizes as
\be
(z-1)^2 (8z+1) = 0.
\ee
The only solutions to this satisfying $z = \cosh^2 r$ are $z = 1$,
since $\cosh r \geq 1$, and the only value of $r$ which satisfies this
is $r = 0$.  Hence the only critical point of $\tV$ is at $r=0$.
\end{proof}

After this conjugation to $\RR^d$, \eqref{eqn:hs} becomes
\begin{eqnarray}
\left\{ \begin{array}{c}
-iu_t - \tilde{\Delta} u +\frac{(d-1)^2}{4} u + V_d(x) u - \tf(x, u)  = 0, \ x \in \RR^d  \\
u(0,x) = u_0 (x) \in H^1,
\end{array} \right.
\label{nls:hs}
\end{eqnarray}
where now the nonlinearity $\tf$ takes the following form after
conjugation:
\be
\tf(x, u ) & = & \phi^{-1} f( \phi u ) (\phi u) \\
& = &   \left( \frac{r}{\sinh r} \right)^{\frac{p(d-1)}{2}}  | u |^p u .
\ee

We have the naturally defined conserved quantities
\begin{eqnarray*}
Q (u) = \| u \|_{L^2}^2
\end{eqnarray*}
and 
\begin{eqnarray}
E(u) = \int_{\RR^d} \left[ \frac{1}{2} | \nabla u |^2 + \frac{1}{2} a(x) | \nabla_{\text{ang}} u |^2 + \frac{1}{2} \left( V_d (|x|) + \frac{(d-1)^2}{4} \right)| u|^2 - F(x,u) \right] dx,
\label{eqn:E}
\end{eqnarray}
where
\begin{eqnarray*}
a(x)  =   \frac{|x|^2 - \sinh^2 |x| }{|x|^2 \sinh^2 |x|} 
\end{eqnarray*}
is the ``offset'' of the spherical Laplacian in the definition of
$\tilde{\Delta}$, 
\begin{eqnarray*}
F(x,u) & = & \int_0^{u(x)} \tf(x,s) ds \\
& = & \frac{1}{p+2} \left( \frac{|x|}{\sinh |x|} \right)^{(d-1)p/2}
|u|^{p+2} \\
& = : & K(|x|)  |u|^{p+2}.
\end{eqnarray*}
From \cite{Ban1} (see Section \ref{S:previous-results}), we have global existence for $p < \frac{4}{d}$ and finite time blow-up for $\frac{4}{d} \leq p < \frac{4}{d-2}$ .

We make a soliton ansatz for \eqref{nls:hs} in $\RR^d$: $u(x,t) = e^{i \lambda t}
R_\lambda$, for a function $R_\lambda$ depending on a real parameter
(the soliton parameter) $\lambda >0$.  Plugging this ansatz into the
conjugated equation \eqref{nls:hs} we see we must have
\begin{eqnarray*}
-\tilde{\Delta} R_\lambda +   \left( \frac{(d-1)^2}{4} + \lambda +
  V_d(r) \right) R_\lambda - \tf (x,R_\lambda) = 0.
\end{eqnarray*}  
Hence, we seek a minimizer of the associated energy functional
\eqref{eqn:E} 
to this nonlinear elliptic equation for $\| u \|_{L^2}$ fixed.

We note that the continuous spectrum is ``shifted'' according to the term
\begin{eqnarray*}
 - \left( \frac{d-1}{2} \right)^2 u.
\end{eqnarray*}
In the end, this term does not alter the existence argument for soliton solutions, however, it does expand the allowed range of soliton parameters from $\lambda \in (0,\infty)$ to 
\begin{eqnarray*}
\lambda \in ( - \left( \frac{d-1}{2} \right)^2, \infty),
\end{eqnarray*}
so henceforward we set
\begin{eqnarray*}
\mu_d = \lambda + \left( \frac{d-1}{2} \right)^2 > 0.
\end{eqnarray*}

\section{Concentration compactness and existence of minimizers}
\label{S:conc-comp}

We recall the celebrated concentration compactness lemma of P.L. Lions, \cite{L1}:

\begin{lem}[Concentration Compactness]
\label{lem:cc}
Let $(\rho_n)_{n \geq 1}$ be a sequence in $L^1 (\RR^d)$ satisfying:
\begin{eqnarray*}
\rho_n \geq 0 \ \text{in $\RR^d$}, \ \int_{\RR^d} \rho_n dx = \delta
\end{eqnarray*}
where $\delta > 0$ is fixed.  Then there exists a subsequence $(\rho_{n_k})_{k \geq 1}$ satisfying one of the three following possibilities:

i. (compactness) there exists $y_k \in \RR^d$ such that $\rho_{n_k} ( \cdot + y_{n_k})$ is tight, i.e.:
\begin{eqnarray*}
\forall \epsilon > 0, \ \exists R < \infty, \ \int_{y_k + B_R} \rho_{n_k} (x) dx \geq \delta - \epsilon;
\end{eqnarray*}

ii. (vanishing) $\lim_{k \to \infty} \sup_{y \in \RR^d} \int_{y + B_R} \rho_{n_k} (x) dx = 0$, for all $R < \infty$; 

iii. (dichotomy) there exists $\alpha \in [0,\delta]$ such that for all $\epsilon > 0$, there exists $k_0 \geq 1$ and $\rho_k^1, \rho_k^2 \in L^1_+ (\RR^d)$ satisfying for $k \geq k_0$:
\begin{eqnarray*}
\left\{ \begin{array}{c}
\| \rho_{n_k} - (\rho_k^1 + \rho_k^2) \|_{L^1} \leq \epsilon, \ | \int_{\RR^d} \rho_k^1 dx - \alpha | \leq \epsilon, \ | \int_{\RR^d} \rho_k^2 dx - (\delta -\alpha) | \leq \epsilon \\
d( \text{Supp} (\rho_k^1), \text{Supp} (\rho_k^2) ) \to \infty.
\end{array} \right. 
\end{eqnarray*}
\end{lem}

We want to apply this in the setting of hyperbolic solitons.  We have
reduced the problem to minimizing energy functionals on $\RR^d$ with
the addition of an angular derivative term and a potential.  However,
we have also seen that any minimizer must be radial, hence the angular
term will vanish.  That means we are left with a minimization problem
with potential on $\RR^d$, for which there is a theory.  We summarize
the basic technique, then we'll indicate how to apply it in the
present setting.
   
To begin, let us look at the basic energy functionals
\begin{eqnarray*}
{\mathcal E}_1 (u) = \int_{\RR^d} \left[ \frac{1}{2} | \nabla u |^2 + \frac{1}{2} c_1(x) |u|^2 - F(x,u) \right] dx,
\end{eqnarray*}
where $F(x,t) = \int_0^t f(x,s) ds$ and $f$ is the nonlinearity and
\begin{eqnarray*}
{\mathcal E}_2 (u) = \int_{\RR^d} \left[ \frac{1}{2} | \nabla u |^2 + \frac{1}{2} c_2(x)| u|^2  \right] dx.
\end{eqnarray*}
Define
\begin{eqnarray*}
I_\delta = \inf \{ {\mathcal E}_1 (u) : u \in H^1 (\RR^d), \| u \|_{L^2}^2 = \delta \}
\end{eqnarray*}
and 
\begin{eqnarray*}
J_\delta = \inf \{ {\mathcal E}_2 (u) : u \in H^1 (\RR^d), \int F(x,u) dx = \delta \}.
\end{eqnarray*}

For $\mathcal{E}_1$, we assume 
\begin{eqnarray}
\label{eqn:E1assump1}
\left\{ \begin{array}{c}
c_1^+ \in L^1_{\text{loc}}, \ \forall \delta > 0, \ c_1^+ 1_{(c_1 \geq \delta)} \in L^{q_1} \ \text{with} \ \max \left\{ \frac{d}{2}, 1 \right\} \leq q_1 < \infty, \\
c_1^- \in L^{q_2} + L^{q_3} \ \text{with} \ \max \left\{ \frac{d}{2}, 1
\right\} \leq q_2, q_3 < \infty
\end{array} \right. 
\end{eqnarray}
and
\begin{eqnarray}
\label{eqn:E1assump2}
\left\{  \begin{array}{c}
f(x,t) \in {\mathcal C} (\RR^d \times \RR), \ f(x,t) \to \bar{f} (t) \ \text{as} |x| \to \infty \ \text{uniformly for $t$ bounded}, \\
\lim_{|t| \to 0} F(x,t) t^{-2} = 0, \ \lim_{|t| \to \infty} F(x,t) |t|^{-d^*} = 0 \ \text{uniformly in $x \in \RR^d$},
\end{array} \right.
\end{eqnarray}
with 
\begin{eqnarray*}
d^* = \frac{2d+4}{d}. 
\end{eqnarray*}

For $J_\delta$, we assume $0 < p < \frac{4}{d-2}$, the integrated
nonlinearity takes the form
\be
F(x,u) = K(x) | u |^{p+2},
\ee
and
\begin{eqnarray}
\label{eqn:E2assump}
\left\{ \begin{array}{c}
c_2^+ \in L^1_{loc}, \ c_2^- \in L^q \ \text{for $\frac{d}{2} \leq q < \infty$}, \ c_2^+ \to \bar{c_2} > 0 \ \text{as $|x| \to \infty$}, \\ 
\exists \nu > 0, \ \forall u \in \mathcal{S}, \ \mathcal{E}_2 (u) \geq \nu \| u \|_{H^1}^2, \\
K \in C (\RR^d), \ K^+ > 0 \ \text{for some $x \in \RR^d$}, \ K \to
\bar{K} \in \reals \ \text{as $|x| \to \infty$}.
\end{array} \right.
\end{eqnarray}

We define
\begin{eqnarray*}
I_\delta^\infty = \inf \{ {\mathcal E}^\infty_1 (u) : u \in H^1 (\RR^d), \| u \|_{L^2}^2 = \delta \},
\end{eqnarray*}
where
\begin{eqnarray*}
{\mathcal E}^\infty_1 (u) = \int_{\RR^d} \left[ \frac{1}{2} | \nabla u |^2  - \bar{F}(u) \right] dx,
\end{eqnarray*}
for
\be
\bar{F} = \int_0^t \bar{f} (s) ds.
\ee
Define $J_\delta^\infty$ in a similar fashion with the convention that if $\bar{K} = 0$, $J_\delta^\infty = \infty$. 
Then, we can state the following theorems due to Lions in \cite{L2}.

\begin{thm}[Lions]
\label{thm:lions1}
The strict subadditivity inequality
\begin{eqnarray}
\label{eqn:sub}
I_\delta < I_\alpha + I^\infty_{\delta-\alpha}, \ \forall \alpha \in [0,\delta)
\end{eqnarray}
is a necessary and sufficient condition for the relative compactness in $H^1$ of all minimizing sequences of $I_\delta$.  In particular, if the subadditivity property holds, there exists a minimum for $I_\delta$.
\end{thm}

\begin{thm}[Lions]
\label{thm:lions2}
Assume \eqref{eqn:E1assump1} and \eqref{eqn:E1assump2} hold.  If $F(x,t) = K(x) |t|^{p_1}$ with $K \in \mathcal{C} (\RR^d)$, $K(x) \to 0$ as $|x| \to \infty$ and $0<p<\frac{4}{d}$, then \eqref{eqn:sub} holds if and only if $I_\delta < 0$.
\end{thm}

\begin{thm}[Lions]
\label{thm:lions3}
Assume \eqref{eqn:E2assump} holds.  For $0<p<\frac{4}{d-2}$, $d \geq 2$, all minimizing sequences of $J_\delta$ are relatively compact in $H^1$ if and only if
\begin{eqnarray*}
J_\delta < J_\delta^\infty.
\end{eqnarray*}
\end{thm}

\begin{rem}
The constrained minimization problem for $J_\delta$ is used to prove
existence of solitons, which explains the ranges of $p$ in Theorem
\ref{thm:main-cm}.  However, we consider also the constrained
minimization problem for $I_\delta$, which is used to prove orbital
stability in Proposition \ref{P:os}.
\end{rem}

We now examine how to apply these theorems in our case.  
For $d \geq
3$, the potential $V_d (|x|) \leq  0$ for all $x$ and $a(x) \leq 0$, hence from the
arguments in Lemma \ref{lem:rad1}, taking a Schwarz symmetrization decreases the energy functionals $\mathcal{E}_1$ and $\mathcal{E}_2$.  Hence, a radial minimizer for $I_\delta$ and $J_\delta$ may be obtained.  

In the case of the hyperbolic soliton equations, we have from above
\begin{eqnarray*}
c_{1} (x) & = & V_d(|x|) \\
& = & - \frac{(d-1)(d-3)}{4} \frac{\sinh^2 r - r^2}{r^2 \sinh^2 r}, \\
c_2(x) & = & V_d(|x|) +\mu_d, \text{ and} \\
F(x,u) & = &   K(|x|)  |u|^{p+2},
\end{eqnarray*}
with
\be
K(|x|) = \frac{1}{p+2} \left( \frac{|x|}{\sinh |x|} \right)^{(d-1)p/2}
\ee
as usual.  Note we have for $d \geq 4$, $c_1 \leq 0$ and
\be
\lambda + \frac{(d-1)^2}{4} \geq c_2 (x) \geq \lambda + \frac{d(d-1)}{6},
\ee
while for $d =3$, $c_1 = 0$ and
$c_2 = \lambda + (d-1)^2/4$, and for $d=2$, $c_1 \geq 0$, and
\be
\lambda + \frac{1}{3} \geq c_2 \geq \lambda + \frac{1}{4}.
\ee 
Hence, we see if $0 < p < 4/d$,
\begin{eqnarray*}
\text{  (i)} && c_{1} \in L^1_{loc}, \ c_1 \in L^q \ \forall q > \frac{d}{2}, \\
\text{ (ii)} && \lim_{t \to 0} F(x,t) t^{-2} = 0, \ \lim_{t \to
  \infty} F(x,t) t^{-d^*} = 0 , \text{ and}\\
\text{(iii)} && f(x,t) \to 0 \ \text{as $|x| \to \infty$ uniformly for $t$ bounded}.
\end{eqnarray*}
Also,
\begin{eqnarray*}
I_\delta^\infty = \inf \{ \frac{1}{2} \int | \nabla u |^2 dx : u \in H^1 (\RR^d), \| u \|_{L^2}^2 = \delta \},
\end{eqnarray*}
so $I_\delta^\infty = 0$ by a simple scaling argument.  As a result, the subadditivity condition becomes
\begin{eqnarray*}
I_\delta < 0.
\end{eqnarray*}
We note that this sub-additivity condition holds for a sufficiently large mass, but is not true in general, see \cite{CMMT}, Appendix $3$.

In the case of $J_\delta$, for the $K$ resulting from hyperbolic
geometry, we have $\bar{K} = 0$ and hence $J_\delta^\infty =
\infty$.  
Also, note that in the notation for $\mathcal{E}_2$, we have
\begin{eqnarray*}
c_2(x) = \mu_d + V_d (x),
\end{eqnarray*}
so the assumptions \eqref{eqn:E2assump} are satisfied provided
$\lambda > (d-1)^2/4$.  Indeed, the only thing to check is that
$\mathcal{E}_2$ is bounded below by $\nu \| u \|_{H^1}^2$ for some
$\nu >0$.  But for $\lambda > (d-1)^2/4$, the set where $c_2 \leq 0$
is bounded, so the lower bound on $\mathcal{E}_2$ follows from the
Galiardo-Nirenberg-Sobolev inequality.

As the case $\bar{K}= 0$ represents an extremal case of the
concentration compactness formulation for $\mathcal{E}_2$, we present the proof here for completeness.

\begin{lem}
There exists a non-trivial minimizer for $J_\delta$.
\end{lem}

\begin{proof}
The fact that compactness implies subadditivity is a standard result of concentration compactness found in \cite{L1}.

Hence, we seek to prove that subadditivity implies compactness.
Let $(u_n)_n$ be a minimizing sequence for $J_\delta$.  As $\| u_n
\|_{H^1}$ is bounded, we have $\| u_n \|_{L^{p+2}}^{p+2} =: \beta_n$
bounded, since $p < 4/(d-2)$.  Then, select a subsequence if necessary such that
\begin{eqnarray*}
\beta_n \to \tilde{\beta} >0.
\end{eqnarray*}
Define $\rho_n = \beta_n^{-1} |u_n|^{p+2}$.  Then,
\begin{eqnarray*}
\int \rho_n dx = 1
\end{eqnarray*}
and we may apply Lemma \ref{lem:cc}.

First, let us rule out the vanishing condition.  If
\begin{eqnarray*}
\sup_{y \in \RR^d} \int_{y + B_R} |\rho_n|^{p+2} dx \to 0 , \ \forall R < \infty,
\end{eqnarray*}
then $\rho_n \to 0$ in $L^\alpha$ for $p+2 < \alpha < 2 + \frac{4}{d-2} =
d^*$ by standard functional analytic arguments from Lemma I.1 in \cite{L2}.  
By assumption, $\| u_n \|_{L^2}$ is bounded for all $n$ hence by
interpolation, $u_n \to 0$ in $L^{p+2}$ and $K u_n^{p+2} \to 0$ in
$L^1$ 
contradicting the constraint.

Next, we must rule out dichotomy.  If such a dichotomy exists, it is clear that either
\begin{eqnarray*}
J_\delta \geq J_{\delta_1}^\infty + J_{\delta_2} = \infty
\end{eqnarray*}
or 
\begin{eqnarray*}
J_\delta \geq J_{\delta_2}^\infty + J_{\delta_1} = \infty,
\end{eqnarray*}
both of which provide a contradiction to the obvious fact that
$J_\delta < \infty$ as well as the sub-additivity condition.

Finally, we have $\exists y_k$ such that $\forall \epsilon > 0$, $\exists R$ such that
\begin{eqnarray*}
\int_{y_k + B_R} \rho_k (x) dx \geq \beta_n - \epsilon.
\end{eqnarray*}
As $\beta_n$ and $\| u_n \|_{H^1}$ are bounded, we have $\| u_n \|_{L^{p+2}}$ bounded, hence if $|y_k| \to \infty$ we get a contradiction to the constraint 
\begin{eqnarray*}
\int K u_n^{p+2} dx = \delta.
\end{eqnarray*}
As a result, $|y_k|$ is bounded.
Hence, $u_n (y_n + \cdot)$ converges strongly in $L^{p+2}$ and weakly in $H^1$ to some $u \in H^1$ satisfying the constraint.  From here, relative compactness follows.

\end{proof}

\section{Soliton properties and stability}
\label{sec:props}

Since we know our soliton is spherically symmetric, we show in this
section that we have exponential decay and $C^2$ smoothness.  These
results follow from the standard ODE and maximum principle 
arguments of \cite{BeLi}, with the superficial modification that our
nonlinearity depends on $r$ but decays at infinity.  This is
summarized in the following Lemmas.

\begin{lem}
If $u$ is a spherically symmetric minimizer of the constrained minimization problems $I_\delta$ or $J_\delta$, then 
\begin{eqnarray*}
u \in C^2 (\RR^d)
\end{eqnarray*}
and
\begin{eqnarray*}
|D^\alpha u | \leq C e^{-\delta |x|}, \ x \in \RR^d
\end{eqnarray*}
for $C, \delta > 0$ and $ |\alpha| \leq 2$.
\end{lem}

\begin{lem}
There exists some $\delta > 0$ such that
\begin{eqnarray*}
|D^\alpha u(x) | \leq C e^{-\delta |x|},
\end{eqnarray*}
for some $C$ and $|\alpha| \leq 2$.
\end{lem}

We now proceed to prove orbital stability of solitons.  
Though we will get existence for any $p < \frac{4}{d-2}$, we will be able to show orbital stability for $p < \frac{4}{d}$ using the arguments of \cite{CL} and the Gagliardo-Nirenberg Inequality.  We will denote by $S$ the set of solutions to the minimization problem $I_\lambda$.

\begin{prop}[Orbital Stability]
\label{P:os}
Given $p < \frac{4}{d}$ and $\lambda > 0$ sufficiently large mass such that $I_\lambda < 0$, for all $\epsilon > 0$, there exists a $\delta >0$ such that if
\begin{eqnarray*}
\inf_{\gamma} \| u_0 (\Omega) - e^{i \gamma} R_\lambda (\Omega ) \|_{H^1 (\HS^d)} < \delta
\end{eqnarray*}
for $R_\lambda$ in $S$
then the corresponding solution $u (\Omega,t)$ of \eqref{nls:hs} satisfies
\begin{eqnarray*}
\inf_{\phi \in S} \| u (\Omega,t) - \phi \|_{H^1 (\HS^d)} < \epsilon.
\end{eqnarray*}
\end{prop}

\begin{proof}[Proof following \cite{CL}] 

Let us take $u_0$ as initial data for \eqref{nls:hs} after conjugation
to Euclidean space.  As mentioned previously, we have the conserved quantities
\begin{eqnarray*}
Q (u) = Q (u_0)
\end{eqnarray*}
and
\begin{eqnarray*}
E (u) = E (u_0).
\end{eqnarray*}
We wish to show
\begin{eqnarray*}
\forall \epsilon > 0, \ \exists \eta > 0, \inf_{\phi \in S} \| u_0 - \phi \|_{H^1} < \eta \Longrightarrow  \inf_{\phi \in S} \| u(t,\cdot) - \phi \|_{H^1} < \epsilon.
\end{eqnarray*}
Assume for a moment that $\| u_0 \|_{L^2} = \lambda$, and assume for
the purpose of contradiction that orbital stability is false.  Then, there exist $\epsilon_0 > 0$, $u_0^n$ and $t_n \geq 0$ such that
\begin{eqnarray*}
\left\{ \begin{array}{c}
u_0^n \in H^1, \ \| u_0^n \|_{L^2} = \lambda, \ E (u_0^n) \to I_\lambda \\
\inf_{\phi \in S} \| u^n (t_n,\cdot) - \phi \|_{H^1} \geq \epsilon_0.
\end{array} \right.
\end{eqnarray*}
However, from the conservation laws we know 
\begin{eqnarray*}
E (u^n (t_n)) \to I_\lambda, \ \| u^n (t_n) \|_{L^2} = \lambda.
\end{eqnarray*}
Hence, $u^n (t_n, \cdot)$ is relatively compact in $H^1$, so
\begin{eqnarray*}
\inf_{\phi \in S} \| u^n (t_n,\cdot) - \phi \|_{H^1}  \to 0
\end{eqnarray*}
and we arrive at a contradiction.

If $\| u_0 \|_{L^2} \neq \lambda$, then repeat the argument with $S$ defined to be the set of solutions to $I_{\| u_0 \|_{L^2}}$ and apply continuity with respect to $\lambda$.

\end{proof}

\begin{rem}
If $R_\lambda$ is shown to be the unique radial minimizer, the resulting statement of the theorem would read
\begin{eqnarray*}
\inf_{\gamma} \| u (\Omega,t) -  e^{i \gamma} R_\lambda (\Omega )  \|_{H^1 (\HS^d)} < \epsilon.
\end{eqnarray*}
\end{rem}

\begin{rem}
There is a much stronger notion of stability referred to as {\it
  asymptotic stability} or {\it scattering} which is given by the
following definition.
\begin{defn}
Let
\begin{eqnarray*}
u_0 (\Omega) = e^{i \gamma_0} R_{\lambda_0} (\Omega) + \phi (\Omega).
\end{eqnarray*}
Then, the corresponding solution $u (\Omega, t)$ of \eqref{nls:hs} is said to be {\it asymptotically stable} if there exists $w(\Omega) \in L^2 (\HS^d)$ and $\sigma_\infty = (\lambda_\infty, \gamma_\infty)$ such that
\begin{eqnarray*}
\lim_{t \to \infty} \| u(\Omega,t) - e^{i \gamma_\infty} R_{\lambda_\infty} (\Omega) - e^{i \frac{t}{2} \Delta_{\HS^d}} w \|_{L^2 (\HS^d)} = 0.
\end{eqnarray*}
\end{defn}
In the present note we do not prove
asymptotic stability, however it is possible to outline the necessary spectral results.  Namely, this sort of stability is proved by linearizing about the soliton, proving dispersive estimates for the resulting skew-symmetric matrix
Hamiltonian operator $\mathcal{H}$, using the modulation parameters to
guarantee orthogonality to any discrete spectrum of $\mathcal{H}$, and
finally doing a standard contraction map on the coupled infinite
dimensional and finite dimension system.  In order to linearize
effectively, we must have a nonlinearity $\beta$ such that
$\beta''(s)$ is bounded for small $s$.  Then, the resulting linearized
operator $\mathcal{H}$ must have a well-understood and well-behaved
spectrum.  

Explicitly, we use the ansatz \begin{eqnarray*}
\psi = e^{i \lambda t}(R_{\lambda} + \phi(x,t)).
\end{eqnarray*}  
For simplicity, set $R = R_\lambda$.  Inserting this into the equation we know that since $\phi$ is a soliton solution we have 
\begin{eqnarray}
i (\phi)_t + \Delta_{\HS^d} (\phi) & = & -\beta( R^2) \phi - 2 \beta'( R^2 )  R^2 \text{Re}(\phi) + O (\phi^2),
\end{eqnarray}
by splitting $\phi$ up into its real and imaginary parts, then doing a Taylor Expansion. Hence, if $\phi = u + iv$, we get
\begin{eqnarray}
\partial_t \left( \begin{array}{c}
u \\
v
\end{array} \right) = \mathcal{H} \left( \begin{array}{c}
u \\
v
\end{array} \right),
\end{eqnarray}
where 
\begin{eqnarray}
\mathcal{H} =  \left( \begin{array}{cc}
0 & L_{-} \\
-L_{+} & 0
\end{array} \right),
\end{eqnarray}
where $$L_{-} = - \Delta_{\HS^d} + \lambda - \beta( R_\lambda )$$ and 
$$L_{+} = - \Delta_{\HS^d} + \lambda - \beta( R_\lambda ) - 2 \beta' (R^2_\lambda) R_\lambda^2.$$
Set $\mu = \lambda + \frac{(d-1)^2}{4}$.
\begin{defn}
\label{spec:defn1}
A Hamiltonian, $\mathcal{H}$ is called admissible if the following hold: \\
1)  There are no embedded eigenvalues in the essential spectrum, \\
2)  The only real eigenvalue in $[-\mu, \mu ]$ is $0$, \\
3)  The values $\pm \lambda$ are not resonances. \\
\end{defn}

\begin{defn}
\label{spec:defn2}
Let (NLS) be taken with nonlinearity $\beta$.  We call $\beta$ admissible if there exists a minimal mass soliton, $R_{min}$, for (NLS) and the Hamiltonian, $\mathcal{H}$, resulting from linearization about $R_{min}$ is admissible in terms of Definition \ref{spec:defn1}.
\end{defn}

The spectral properties we need for the linearized Hamiltonian equation in order to prove stability results are precisely those from Definition \ref{spec:defn1}.
\end{rem}

\section{Other nonlinearities}
\label{S:other-nl}
Note that much of the analysis above is in a sense simpler than in the
Euclidean case because in our energy functional, the potentials and
coefficients involve terms which decay at spatial infinity, and the
nonlinearity, once conjugated to Euclidean space, decays
exponentially.  Hence if we have a nonlinearity of the form
\begin{eqnarray}
\label{E:nl-grow}
g(\Omega) |u|^p u,
\end{eqnarray}
where $g(\Omega)$ is now allowed to grow exponentially at a rate
slower than
\be
\left( \frac{\sinh r}{r} \right)^{(d-1)p/2},
\ee
it is actually closer to the Euclidean case.  In other words, our
techniques extend trivially to show solitons exist with extremely
powerful nonlinearities, growing exponentially at spatial infinity.
It is unclear then whether simple power nonlinearities as in
\eqref{eqn:hs}, or exponentially growing nonlinearities as in
\eqref{E:nl-grow} are more ``physical'', as they resemble the
Euclidean case more.

Saturated nonlinearities are of the form
\begin{eqnarray}
\label{sat:eqn1}
f (s) = s^{q} \frac{s^{p-q}}{1 + s^{p-q}}, 
\end{eqnarray}
where $p > 2+\frac{4}{d}$ and $\frac{4}{d}  > q > 0$ for $d \geq 3$ and $\infty > p > 2+\frac{4}{d} > \frac{4}{d} > q > 0$ for $d < 3$.

\begin{rem}
For $|u|$ large, the behavior is $L^2$ subcritical and for $|u|$ small, the behavior is $L^2$ supercritical.  For the case of asymptotic stability, $p$ is chosen much larger than the $L^2$ critical exponent, $\frac{4}{d}$ in order to allow sufficient regularity when linearizing the equation.  
\end{rem}

Then, upon conjugation by $\phi$, we have
\begin{eqnarray*}
f(x,|u|) = \phi^{q-1} \frac{|u|^{p-1}}{\phi^{q-p} + |u|^{p-q}}. 
\end{eqnarray*}
Since
\begin{eqnarray*}
\phi^{q-p} > 1 
\end{eqnarray*}
and
\begin{eqnarray*}
\lim_{x \to \infty} \phi^{q-p} (x) = \infty,
\end{eqnarray*}
using similar techniques to those above we may prove similar soliton existence results for all $\mu$. However, similar to the Euclidean study of saturated nonlinearities, solitons for $\mu$ large will be stable and solitons for $\mu$ small will be unstable.

\bibliographystyle{alpha}
\bibliography{CM-bib}

\def\cprime{$'$} \def\cprime{$'$}
  \def\cftil#1{\ifmmode\setbox7\hbox{$\accent"5E#1$}\else
  \setbox7\hbox{\accent"5E#1}\penalty 10000\relax\fi\raise 1\ht7
  \hbox{\lower1.15ex\hbox to 1\wd7{\hss\accent"7E\hss}}\penalty 10000
  \hskip-1\wd7\penalty 10000\box7}
\begin{thebibliography}{CMMT14}

\bibitem[AP08]{AnPi}
J.-P. Anker and V.~Pierfelice.
\newblock Nonlinear schr{\"o}dinger equation on real hyperbolic spaces.
\newblock {\em preprint}, 2008.

\bibitem[Ban07]{Ban1}
V.~Banica.
\newblock The nonlinear {S}chr\"odinger equation on hyperbolic space.
\newblock {\em Comm. Partial Differential Equations}, 32(10-12):1643--1677,
  2007.

\bibitem[BCD09]{BCD}
V.~Banica, R.~Carles, and T.~Duyckaerts.
\newblock On scattering for nls: from euclidean to hyperbolic space.
\newblock {\em Disc. Contin. Dyn. Syst.}, 24(4):1113--1127, 2009.

\bibitem[BCS08]{BCS}
V.~Banica, R.~Carles, and G.~Staffilani.
\newblock Scattering theory for radial nonlinear {S}chr\"odinger equations on
  hyperbolic space.
\newblock {\em Geom. Funct. Anal.}, 18(2):367--399, 2008.

\bibitem[BL83]{BeLi}
H.~Berestycki and P.-L. Lions.
\newblock Nonlinear scalar field equations. {I}. {E}xistence of a ground state.
\newblock {\em Arch. Rational Mech. Anal.}, 82(4):313--345, 1983.

\bibitem[CL82]{CL}
T.~Cazenave and P.-L. Lions.
\newblock Orbital stability of standing waves for some nonlinear
  {S}chr\"odinger equations.
\newblock {\em Comm. Math. Phys.}, 85(4):549--561, 1982.

\bibitem[CMMT14]{CMMT}
Hans Christianson, Jeremy Marzuola, Jason Metcalfe, and Michael Taylor.
\newblock Nonlinear bound states on weakly homogeneous spaces.
\newblock {\em Communications in Partial Differential Equations}, 39(1):34--97,
  2014.

\bibitem[DM88]{DM}
E.~B. Davies and N.~Mandouvalos.
\newblock Heat kernel bounds on hyperbolic space and {K}leinian groups.
\newblock {\em Proc. London Math. Soc. (3)}, 57(1):182--208, 1988.

\bibitem[Dra05]{Dra}
C.~Draghici.
\newblock Rearrangement inequalities with application to ratios of heat
  kernels.
\newblock {\em Potential Anal.}, 22(4):351--374, 2005.

\bibitem[GSS90]{GSS}
M.~Grillakis, J.~Shatah, and W.~Strauss.
\newblock Stability theory of solitary waves in the presence of symmetry. {II}.
\newblock {\em J. Funct. Anal.}, 94(2):308--348, 1990.

\bibitem[IS08]{IS}
A.~Ionescu and G.~Staffilani.
\newblock Semilinear {S}chr{\"o}dinger flows on hyperbolic spaces: scattering
  in ${H}^1$.
\newblock {\em preprint}, 2008.

\bibitem[Lio84a]{L1}
P.-L. Lions.
\newblock The concentration-compactness principle in the calculus of
  variations. {T}he locally compact case. {I}.
\newblock {\em Ann. Inst. H. Poincar\'e Anal. Non Lin\'eaire}, 1(2):109--145,
  1984.

\bibitem[Lio84b]{L2}
P.-L. Lions.
\newblock The concentration-compactness principle in the calculus of
  variations. {T}he locally compact case. {II}.
\newblock {\em Ann. Inst. H. Poincar\'e Anal. Non Lin\'eaire}, 1(4):223--283,
  1984.

\bibitem[LL01]{LL}
E.~H. Lieb and M.~Loss.
\newblock {\em Analysis}, volume~14 of {\em Graduate Studies in Mathematics}.
\newblock American Mathematical Society, Providence, RI, second edition, 2001.

\bibitem[McL93]{Mc}
K.~McLeod.
\newblock Uniqueness of positive radial solutions of {$\Delta u+f(u)=0$} in
  {${\bf R}\sp n$}. {II}.
\newblock {\em Trans. Amer. Math. Soc.}, 339(2):495--505, 1993.

\bibitem[Pan92]{Pank}
A.~A. Pankov.
\newblock Invariant semilinear elliptic equations on a manifold of constant
  negative curvature.
\newblock {\em Funktsional. Anal. i Prilozhen.}, 26(3):82--84, 1992.

\bibitem[Sha83]{Sh1}
J.~Shatah.
\newblock Stable standing waves of nonlinear {K}lein-{G}ordon equations.
\newblock {\em Comm. Math. Phys.}, 91(3):313--327, 1983.

\bibitem[Sha85]{Sh2}
J.~Shatah.
\newblock Unstable ground state of nonlinear {K}lein-{G}ordon equations.
\newblock {\em Trans. Amer. Math. Soc.}, 290(2):701--710, 1985.

\bibitem[SS85]{ShSt}
J.~Shatah and W.~Strauss.
\newblock Instability of nonlinear bound states.
\newblock {\em Comm. Math. Phys.}, 100(2):173--190, 1985.

\bibitem[SS99]{SS}
C.~Sulem and P.-L. Sulem.
\newblock {\em The nonlinear {S}chr\"odinger equation}, volume 139 of {\em
  Applied Mathematical Sciences}.
\newblock Springer-Verlag, New York, 1999.
\newblock Self-focusing and wave collapse.

\bibitem[Wei85]{W1}
M.~I. Weinstein.
\newblock {Modulational stability of ground states of nonlinear Schr\"odinger
  equations}.
\newblock {\em SIAM J. Math. Anal.}, 16:472--491, 1985.

\bibitem[Wei86]{W2}
M.~I. Weinstein.
\newblock {Lyapunov stability of ground states of nonlinear dispersive
  evolution equations}.
\newblock {\em Comm. Pure Appl. Math.}, 39:472--491, 1986.

\end{thebibliography}

\end{document}